\documentclass[12pt,reqno]{amsart}
\usepackage[margin=1.2in]{geometry}
\usepackage{amsmath,amssymb,amsfonts,amsthm}
\usepackage{mathrsfs}
\usepackage{marginnote}
\usepackage{enumitem}
\usepackage{booktabs}
\usepackage{graphicx}
\usepackage{url}
\usepackage{xcolor}
\usepackage{verbatim}
\usepackage[colorlinks=true,linkcolor=blue,citecolor=red,filecolor=blue,urlcolor=cyan]{hyperref}
\usepackage{tikz}
\usetikzlibrary{arrows.meta,calc,angles,quotes}

\numberwithin{equation}{section}

\newtheorem{theorem}{Theorem}[section]
\newtheorem{lemma}[theorem]{Lemma}
\newtheorem{proposition}[theorem]{Proposition}
\newtheorem{example}[theorem]{Example}

\newtheorem{remark}[theorem]{Remark}

\newenvironment{nouppercase}{%
	\renewcommand{\uppercasenonmath}[1]{}}{}

\def\geq{\geqslant}
\def\leq{\leqslant}
\def\ge{\geqslant}
\def\le{\leqslant}
\newcommand{\dd}{\,\mathrm{d}}
\newcommand{\E}{\mathbb{E}}
\newcommand{\R}{\mathbb{R}}

\DeclareMathOperator{\Ent}{Ent}
\DeclareMathOperator{\Var}{Var}
\DeclareMathOperator{\Cov}{Cov}
\DeclareMathOperator{\tr}{tr}
\allowdisplaybreaks

\begin{document}
\title[Quantitative bounds for high-dimensional entropic CLT]{\LARGE Quantitative bounds for high dimensional entropic CLT}

\author[C.-S.~Deng]{Chang-Song Deng}
\address[C.-S.~Deng]{School of Mathematics and Statistics\\ Wuhan University\\ Wuhan 430072, China}
\email{dengcs@whu.edu.cn}

\author[L.~Wang]{Lin Wang}
\address[L.~Wang]{School of Mathematical Sciences, University of Chinese Academy of Sciences, Beijing 100049,
China; Academy of Mathematics and Systems Science, Chinese Academy of Sciences, Beijing 100190, China}
\email{wanglin2021@amss.ac.cn}

\author[L.~Xu]{Lihu Xu}
\address[L.~Xu]{Lihu Xu, 1. Department of Statistics and Probability, Michigan State University, East Lansing, MI 48824, USA; 2. Department of Mathematics,  Faculty of Science and Technology,  University of Macau,  Av. Padre Tom\'as Pereira,  Taipa,  Macau,  China}
\email{xulihu@msu.edu, lihuxu@um.edu.mo}

	\begin{abstract}
	By extending the Johnson--Barron projection method from one dimension to high dimensions and utilizing a Wang type dimension-free Harnack inequality, we obtain a new quantitative bound for the entropic central limit theorem under the assumption that the Poincar\'e inequality holds. We compare our results with recent developments to demonstrate the merits of our approach.	
	\end{abstract}
	
	\date{\today}

\begin{nouppercase}
	\maketitle
\end{nouppercase}
	\tableofcontents

	\section{Introduction}

Throughout this paper, we denote by $p_X$ the density (if it exists) of a random variable $X$ on $\R^d$.
Let $(X_i)_{i\geq 1}$ be a sequence of  independent and identically distributed (i.i.d.) $\R^d$-valued random vectors with $\E X_i=0$ and $\Cov(X_i)=I_d$, and set
\[
Z_n:=\frac{1}{\sqrt n}\sum_{i=1}^n X_i,
\qquad
Z\sim\mathcal N(0,I_d).
\]
Here and in what follows, $I_d$ denotes the $d\times d$ identity matrix. When $\E X_i=\mu \in \R^d$ and $\Cov(X_i)=\Sigma$ is invertible, we define $\tilde X_i=\Sigma^{-\frac 12}(X_i-\mu)$ so that $\tilde X_i$ satisfies the conditions above.

 By the classical central limit theorem (CLT), $Z_n$ converges weakly to $Z$ as $n\to\infty$. In this work, we study the entropic CLT, namely, the decay of the relative entropy of $Z_n$ with respect to $Z$.  Recall that the relative entropy of a random variable $X$ on $\mathbb{R}^d$ with respect to $Z$ is defined by
\[
\Ent(X\mid Z):=\int_{\mathbb{R}^d} p_X(x)\log\frac{p_X(x)}{p_Z(x)}\dd x,
\] 
whenever the density $p_X$ exists; otherwise, we adopt the convention that $\Ent(X\mid Z)=\infty$. We aim to establish quantitative bounds for $\Ent(Z_n\mid Z)$ with explicit dependence on both $n$ and the dimension $d$.

In the one-dimensional case, quantitative bounds of order $O(n^{-1})$ for the relative entropy were obtained in \cite{ABBN04,JB04} under a spectral gap (Poincar\'e inequality) assumption. Recall that the Poincar\'e constant of an $\mathbb{R}^d$-valued random 
variable $X$ is defined as 
$$
C_P(X)=\sup\frac{\operatorname{Var}(f(X))}{\mathbb{E}\left[|\nabla f(X)|^2\right]},
$$
where the supremum is taken over all locally Lipschitz functions $f:\mathbb{R}^d\to\mathbb{R}$. See Subsection \ref{WangHI} below for more details on the Poincar\'e (or spectral gap) inequality.

In the multi-dimension case, several quantitative results were recently developed. 
Bobkov, Chistyakov and G\"otze \cite{BCG13} derive Edgeworth-type asymptotic expansions for the entropic distance $\Ent(Z_n\mid Z)$. For fixed dimension $d$, they established the optimal rate of $O(n^{-1})$ under the finite fourth moment assumption, to the best of our knowledge, this condition is the best known one to guarantee the $O(n^{-1})$ rate.   

 In the recent work \cite{EMZ20}, Eldan, Mikulincer and Zhai introduced a novel martingale embedding (Skorokhod embedding) construction  to obtain high‑dimensional CLTs. They assumed that $\E X_1=0$, $\Cov(X_1)=I_d$ and $X_1$ is log‑concave, i.e., the density $p_{X_1}$ of $X_1$ satisfies 
that $\nabla^2\log p_{X_1}(x)$ is non-positive definite 
for all $x\in\mathbb{R}^d$. Under these conditions, they established the non‑asymptotic convergence rate in relative entropy (\cite[Theorem 6]{EMZ20}):
\begin{equation}\label{entro66}
	\Ent(Z_n \mid Z)
	\le C\frac{d^{10}\bigl(1 + \Ent(X_1\mid Z)\bigr)}{n},
\end{equation}
 where $C>0$ is a universal constant independent of $n$ and $d$. Since for a log‑concave random variable $X_1$ on $\R^d$ one has $\Ent(X_1\mid Z)=O(d)$ (see, e.g., \cite{{KL25}}), the overall bound is of order $O(d^{11}/n)$, i.e., it exhibits a high‑degree polynomial dependence on the dimension. The approach in \cite{EMZ20} is based on the entropy minimizing process (\cite{Fol84}, \cite{Fol86}) and the stochastic localization (\cite{Eld13}), which allow them to control the fluctuations of the quadratic variation process, and they apply Girsanov’s theorem to translate these fluctuations into entropy bounds. However, the dependence on the dimension may be far from optimal in the general log‑concave case. If, furthermore, $X_1$ is strongly log-concave (i.e., there exists $c>0$ such that $-\nabla^2\log p_{X_1}(x)-c I_d$ is non-negative definite 
for all $x\in\mathbb{R}^d$), they present the rate of order $O(d/n)$ in \cite[Theorem 7]{EMZ20}.

In the present work, we exploit the entropy dissipation along the Ornstein--Uhlenbeck flow, combined with relative Fisher information estimates and functional inequalities, to establish explicit high-dimensional convergence rates for $\Ent(Z_n\mid Z)$ under the assumption of a Poincar\'e inequality. In particular, for certain classes of distributions such as log-concave measures, our results yield sharper rates than previously known rates in the high-dimensional regime, see \eqref{eq-log-concave} below for the detail. 

Throughout this paper, we denote by $C$ a positive constant independent of $d$ and $n$ that may vary from line to line.

	\subsection{Main result and comparison with existing bounds}
For an $\mathbb{R}^d$-valued random vector $X$  satisfying $\mathbb{E}X=0$ and $\Cov(X)=I_d$, the relative Fisher information of $X$ with respect to the distribution of the $d$-dimensional standard Gaussian random variable $Z$ is
\begin{equation} \label{def-reFI}
J(X):=
\int_{\mathbb{R}^d}
\bigl|\nabla\log p_X(x)-\nabla\log p_Z(x)\bigr|^2p_X(x)\dd x.
\end{equation}

	\begin{theorem}\label{thm-1}
		Let $X_1,X_2,\dots,X_n$ be i.i.d.\ $\R^d$-valued random vectors with absolutely continuous density,
		$\E X_1=0$, $\Cov(X_1)=I_d$, and Poincar\'e constant $C_P(X_1)=R$. 
		Let $Z_n=\frac1{\sqrt n}\sum_{i=1}^n X_i$ and $Z\sim\mathcal N(0,I_d)$. Then for all $n\geq1$,
		\[
		\Ent(Z_n\mid Z)\leq \min\left\{\frac{2dR}{2dR+(n-1)}\Ent(X_1\mid Z), \,\sqrt{\frac{d(R-1)}{n}\cdot
		\frac{2dR}{2dR+(n-1)}J(X_1)}\right\}.
		\]
	\end{theorem}
	\begin{remark}\label{comparem}
    We make three comments to compare our Theorem \ref{thm-1}  with existing results.
    \begin{enumerate}[label={\normalfont(\arabic*)}]
        \item   If $\Ent (X_1\mid Z)<\infty$ or $J(X_1)<\infty$, 
        Theorem \ref{thm-1} presents a rate of $O(n^{-1})$ for fixed $d$, which coincides with the optimal rate established in \cite[Theorem 6.1]{BCG13}. However, as noted in \cite{EMZ20}, the estimates in \cite{BCG13} are not explicit in $d$,  and the dependence there seems to be exponential in the dimension $d$.
        By contrast, our result implies an explicit polynomial dependence on the dimension under the log-concavity assumption, see \ref{compadfh3} below.
          
        \item If we further assume that  $X_1$ is  log-concave, then according to \cite[Equation (1.4) and Theorem 1.2]{Kla23}, the Poincar\'e constant of $X_1$ satisfies
	\[
	R=C_P(X_1)\le C\log (d+1).
	\]
	This, together with Theorem \ref{thm-1}, implies that
	\begin{equation*}
\begin{split}
	\Ent(Z_n\mid Z)
	&\le \frac{2dR}{2dR+(n-1)}\Ent(X_1\mid Z)\\
	&\le \frac{2dC\log (d+1)}{2dC\log (d+1)+(n-1)}\Ent(X_1\mid Z),
	\end{split}
    \end{equation*}
	which can be rewritten as
	\begin{equation}\label{eq-log-concave}
	    \Ent(Z_n\mid Z)\le
	C\frac{d\log (d+1)}{n+d\log (d+1)}\Ent(X_1\mid Z).
	\end{equation}
	This provides a better dependence on the dimension $d$ than the result \eqref{entro66} in  \cite[Theorem 6]{EMZ20} under the same assumptions.

\item\label{compadfh3} We still assume that  $X_1$ is  log-concave. Then, from \cite[Equation (3) and Theorem 1.2]{KL25}, it holds that $\Ent(X_1\mid Z)=O(d)$. Combining this with \eqref{eq-log-concave} yields an upper bound of order $O(1/n)$ for fixed $d$ and $O(d)$ for fixed $n$. This is consistent with the upper bound $O(d/n)$ as in \cite[Theorem 7]{EMZ20}, but without requiring the stronger condition that $X_1$ is strongly log-concave.
\item Another non‑asymptotic high‑dimensional bounds have been obtained under functional Poincar\'e inequalities and the Stein kernels argument (\cite{CFP19}):
\begin{equation}\label{steinentro77}
\Ent(Z_n \mid Z)
\le\frac{d(R-1)}{2n}
\log\left(1 + \frac{ J(X_1)n}{(R-1)d}\right),
\end{equation}
 where $R=C_P(X_1)$ is the Poincar\'e constant and $J(X_1)$ is the relative Fisher information of $X_1$ with respect to the standard Gaussian measure (see \eqref{def-reFI} below for the precise definition). When $J(X_1)$ is finite, the upper bound for $\Ent(Z_n \mid Z)$ involves an additional $\log n$ factor, leading to a convergence rate of $O(n^{-1}\log n)$ for fixed $d$. 
Two crucial tools in \cite{CFP19} are the so-called HSI inequality (cf. \cite[(1.3)]{LNP15}) and bounds for the Stein discrepancy derived from the Poincar\'e inequality.
\end{enumerate}
\end{remark}

\subsection{Proof strategy}

The proof of Theorem \ref{thm-1} is based on a decomposition of the entropy along the Ornstein--Uhlenbeck (OU) semigroup, with a time parameter $t>0$ that will later be chosen so as to balance the two terms and optimize the bound.

Let $Z\sim\mathcal N(0,I_d)$ and  $X$ be an independent random vector on $\R^d$ with $\mathbb{E}X=0$, $\operatorname{Cov}(X)=I_d$, and $\Ent(X\mid Z)<\infty$. For $t\ge0$, we define
\[
X(t):=e^{-t}X+\sqrt{1-e^{-2t}}Z.
\]
One has the entropy dissipation identity (see Lemma \ref{lem:debruijn} below)
\begin{equation}\label{eq:entropy-dissipation}
	\Ent(X\mid Z)=\Ent(X(t)\mid Z)
	+\int_0^t J(X(s))\dd s,
\end{equation}
so that $\Ent(X\mid Z)$ can be decomposed into a ``late-time'' entropy
$\Ent(X(t)\mid Z)$ and an integral of the relative Fisher information $J(X(s))$ over $[0,t]$.

To bound $\Ent(X(t)\mid Z)$, we first note that the distribution of $X(t)$ coincides with that of $X_t$, which solves the stochastic differential equation
$$
\dd X_t=-X_t\dd t+\sqrt2\dd B_t,\qquad X_0=X,
$$
where $B_t$ is a standard Brownian motion on $\R^d$ independent of $X$. Then it follows from the entropy-cost inequality \eqref{entro112} below (derived from the Wang type dimension free log-Harnack inequaility) that, for all $t>0$,
\begin{equation}\label{eq:Ent-OU-W2-intro}
	\Ent(X(t)\mid Z)\le \frac{1}{2(e^{2t}-1)}W_2^2(X, Z).
\end{equation}
Here and in the following, $W_2(X, Z)$ denotes the Wasserstein-2 distance between the distributions of $X$ and $Z$.
On the other hand, the accumulated relative Fisher information term $\int_0^t J(X(s))\dd s$ is controlled by the exponential decay of relative Fisher information along the OU flow (see Lemma \ref{lem:decay-Jt} below):
\begin{equation}\label{eq:FI-decay-OU}
	J(X(t))\le e^{-2t}J(X),\quad t\ge0.
\end{equation}
Substituting \eqref{eq:Ent-OU-W2-intro} and \eqref{eq:FI-decay-OU} into \eqref{eq:entropy-dissipation} yields, for every $t>0$,
\begin{equation}\label{eq:Ent-W2-FI-intro}
	\Ent(X\mid Z)
	\le \frac{1}{2(e^{2t}-1)}W_2^2(X,Z)
	+ \frac{1-e^{-2t}}{2}J(X).
\end{equation}
In the proof of Theorem \ref{thm-1}, we apply \eqref{eq:Ent-W2-FI-intro} with
$X=Z_n$, and choose $t$ as a function of $W_2(Z_n,Z)$ and $J(Z_n)$ to balance the two terms on the right hand side of \eqref{eq:Ent-W2-FI-intro}.
The choice $t=\infty$ formally removes the first term and recovers the
Gaussian log-Sobolev inequality
\[
\Ent(X\mid Z)\le \frac12 J(X),
\]
while the choices of $t<\infty$ allow us to exploit corresponding bounds on $W_2(Z_n, Z)$.
Finally, to complete the quantitative relative entropy estimate for $Z_n$, we require an explicit bound on $J(Z_n)$ in the high dimension case, which is proved in Proposition \ref{Prop-FI} by extending the Johnson-Barron projection method from one dimension to high dimensions.

	\subsection{Organization of the paper} The rest of this paper is organized as follows. In Section \ref{sec-pre}, we collect the preliminaries used throughout the paper. Subsection \ref{susec-FI} recalls basic facts on the score function, Fisher information, relative Fisher information, and the de Bruijn identity. Subsection \ref{WangHI} presents the functional inequalities used later, including the Wang type dimension free log-Harnack inequality for the OU semigroup \cite{RW10}, the associated entropy-cost inequality, and basic properties of the Poincar\'e inequality. Subsection \ref{subsec-space} introduces the ridge and additive subspaces together with the projection identities.
Section \ref{sec-main} is devoted to the proof of the main results. In Subsections \ref{subsec-lowerbound} and  \ref{subsec-FI}, we derive the basic lower bounds arising from the projection error decomposition and develop a high-dimensional extension of the Johnson--Barron projection method \cite{JB04}, which yields the key estimate for the decay of $J(Z_n)$. Finally, Subsection \ref{subsec-ent} combines this estimate with the entropy dissipation along the OU semigroup and the Wasserstein-2 bound to complete the proof of Theorem \ref{thm-1}.

	\section{Preliminaries}\label{sec-pre}

\subsection{Fisher information, relative Fisher information, and de Bruijn identity}\label{susec-FI}

Let $X$ be an $\R^d$-valued random vector with  density $p_X$ with respect
to the Lebesgue measure. Assume $\E X=0$ and $\E[|X|^2]<\infty$, and let
\[
\Sigma=\Cov(X)=\E[XX^\top]
\]
be positive definite.

Throughout this paper, the (vector) \emph{score function} of a random variable $X$ is defined by
\[
\rho_X(x):=\nabla\log p_X(x),\qquad x\in\R^d.
\]
For any $c\neq 0$, the score function satisfies the scaling property
\[
\rho_{cX}(x)=c^{-1}\rho_X(x/c),\quad x\in\mathbb R^d.
\]
The \emph{Fisher information} and the \emph{Fisher information matrix} of $X$ are
\begin{equation} \label{e:FIMatrix}
I(X):=\E\big[|\rho_X(X)|^2\big],
\qquad
I_M(X):=\E\bigl[\rho_X(X)\rho_X(X)^\top\bigr].
\end{equation}

We first state some properties of the score function.
\begin{lemma}[{\cite[Lemma A.1]{JB04}}]\label{lem:score-shift}
	Let $X$ be as above and assume that $I(X)<\infty$.
	For an open set $U\subset\R^d$, let $f:\R^d\to\R$ be a measurable function such
	that
	\[
	\sup_{u\in U}\E\bigl[f(u+X)^2\bigr]<\infty.
	\]
	Define
	\[g(u):=\E\bigl[f(u+X)\bigr],\qquad u\in U.
	\]
	Then $g$ is absolutely continuous on $U$, and for almost every
	$u\in U$, one has
	\begin{equation}\label{eq:score-shift-derivative}
		\nabla g(u)= -\E\bigl[f(u+X)\rho_X(X)\bigr].
	\end{equation}
\end{lemma}

\begin{lemma}\label{lem:score-moments}
	Suppose $X$ satisfies the assumptions of Lemma \ref{lem:score-shift}.
	Then
	\[
	\E[\rho_X(X)]=0,
	\qquad
	\E\bigl[X\rho_X(X)^\top\bigr]=-I_d.
	\]
\end{lemma}

\begin{proof}
	Apply Lemma \ref{lem:score-shift} with $f\equiv1$ and $U=\R^d$.  Then
 $\nabla g(u)=0$ for all $u$, and
	\eqref{eq:score-shift-derivative} gives $\E[\rho_X(X)]=0$.
	Next, write $x=(x^{(1)},\dots,x^{(d)})$. For each $j\in\{1,\dots,d\}$, by taking $f(x)=x^{(j)}$, we get 
\[
g(u)=\E\bigl[X^{(j)}+u^{(j)}\bigr]=u^{(j)}.
\]   Evaluating
	\eqref{eq:score-shift-derivative} at $u=0$ yields
 $\E[X\rho_X(X)^\top]=-I_d$.
\end{proof}

	The \emph{relative Fisher information} of $X$ with respect to $\mathcal{N}(0,\Sigma)$ is
	\begin{equation*}
\begin{split}
J(X)&:=\E\left[\bigl|\Sigma^{1/2}\bigl(\rho_X(X)+\Sigma^{-1}X\bigr)\bigr|^2\right]\\
&=\int_{\mathbb{R}^d}
\bigl|\Sigma^{1/2}(\nabla\log p_X(x)-\nabla\log\phi_\Sigma(x))\bigr|^2p_X(x)\dd x,
\end{split}
\end{equation*}
where $\phi_\Sigma$ is the density function of $\mathcal N(0,\Sigma)$. When $\Sigma=I_d$, the above definition coincides with \eqref{def-reFI}.

Using Lemma \ref{lem:score-moments} and $\E[XX^\top]=\Sigma$, we can rewrite $J(X)$ as
\begin{equation*}
	J(X)= I\bigl(\Sigma^{-1/2}X\bigr)-d.
\end{equation*}
In particular, if $\Cov(X)=I_d$, then $J(X)=I(X)-d$.
Moreover, $J$ is invariant under invertible linear changes of variables: if $Y=AX$ for an invertible matrix $A$, then $J(Y)=J(X)$.

Let $X$ be an $\R^d$-valued random vector with zero mean and positive definite covariance matrix $\Sigma$, and let $Z_\Sigma\sim\mathcal N(0,\Sigma)$ be a centered Gaussian random variable with the same covariance.
Set $\widetilde X:=\Sigma^{-1/2}X$ and
$Z:=\Sigma^{-1/2}Z_\Sigma\sim\mathcal N(0,I_d)$. Since relative entropy is invariant under invertible linear transformations, we have
\begin{equation*}
	\Ent(X\mid Z_\Sigma)=\Ent(\widetilde X\mid Z).
\end{equation*}
Thus, for the entropic CLT, it suffices to work in the isotropic case
$\Cov(X)=I_d$.

From now on, we assume $\E X=0$,  $\Cov(X)=I_d$ and fix
$Z\sim\mathcal N(0,I_d)$ independent of $X$.
\begin{lemma}\label{lem:decay-Jt}
For  $t\ge0$, let   \[
X(t):=e^{-t}X+\sqrt{1-e^{-2t}}Z.
\] Then $$
J(X(t))\le e^{-2t}J(X).
$$
\end{lemma}
\begin{proof}
    We prove this lemma using Stam's inequality and the scaling property of Fisher information.	By Stam's inequality \cite{Stam59, Bla65}, for any independent
	$\R^d$-valued random vectors $Y_1,Y_2$ with finite Fisher information,
	\[
	\frac1{I(Y_1+Y_2)} \ge \frac1{I(Y_1)}+\frac1{I(Y_2)}.
	\]
	Apply this with
	\[
	Y_1=e^{-t}X,\qquad Y_2=\sqrt{1-e^{-2t}}Z,
	\]
	so that $Y_1+Y_2=X(t)$. Using the scaling rule
	$I(aY)=a^{-2}I(Y)$ , we obtain
	\begin{equation*}
		\frac1{I(X(t))}
		 \ge \frac1{I(e^{-t}X)}+\frac1{I(\sqrt{1-e^{-2t}}Z)}= \frac{e^{-2t}}{I(X)}+\frac{1-e^{-2t}}{d}.
	\end{equation*}
Combining with the weighted harmonic-arithmetic means inequality,
\[
I\bigl(X(t)\bigr) \le \frac{1}{\frac{e^{-2t}}{I(X)}+\frac{1-e^{-2t}}{d}}\le d + e^{-2t}\bigl(I(X)-d\bigr).
\]
Therefore,
	\[
	J\bigl(X(t)\bigr)
	=I\bigl(X(t)\bigr)-d\leq e^{-2t}J(X).\qedhere
	\]
\end{proof}
The above proof can be viewed as a probabilistic explanation of the
	exponential decay of relative Fisher information along the OU semigroup.
	An alternative derivation uses the geometric interpretation of entropy and
	relative Fisher information along diffusion flows developed by Otto and Villani
	(see, e.g., \cite{Otto01}), where the OU flow is seen as a gradient flow of the relative entropy.

\begin{lemma}[de Bruijn identity]\label{lem:debruijn}
	Assume that $\Ent(X\mid Z)<\infty$.
	Then, for every $t>0$,
$$
		\Ent(X\mid Z)-\Ent(X(t)\mid Z)
		=\int_0^t J(X(s))\dd s.
	$$
\end{lemma}
This is the classical de Bruijn identity; the proof follows from
\cite[Lemma 1]{Bar86}, \cite[Theorem 17.7.2]{CT06} and the change of variables
$s=e^{2t}-1$, and we omit the details.

\subsection{Functional inequalities: Wang type dimension-free Harnack inequality and Poincar\'e inequality}\label{WangHI}
As a weaker version of Wang type dimension-free Harnack inequality put forward in \cite{Wan97}, Wang type log-Harnack inequality \cite{RW10} has become an efficient tool on stochastic analysis; for an in-depth explanation of its applications, we refer to the monograph by Wang 
\cite[Subsection 1.4.1]{Wan13} and the references therein.

Consider the OU process
\begin{equation}\label{eq-OUSDE11}
\dd X_t=-X_t\dd t+\sqrt2\dd B_t,
\end{equation}
where $B_t$ is a standard Brownian motion on $\R^d$. Denote by $P_t$ the Markov semigroup associated with \eqref{eq-OUSDE11}. Using coupling by change of measure (see e.g. \cite[pp. 4-5]{Wan13}), we have the log-Harnack inequality: for every bounded measurable function $f:\R^d\to[1,\infty)$,
$$
P_t\log f(x)\le\log P_tf(y)+\frac{1}{2\left(e^{2t}-1\right)}|x-y|^2,\quad t>0,\,x,y\in\R^d.
$$
Let $Z\sim\mathcal N(0,I_d)$ and  $X$ be an independent random vector on $\R^d$ admitting a density. For simplicity, denote by $X_t$ the solution to \eqref{eq-OUSDE11} with initial value $X_0=X$. 
By \cite[Theorem 1.4.2\,(3)]{Wan13}, the log-Harnack inequality above implies the entropy-cost inequality
\begin{equation}\label{entro112}
\Ent(X_t\mid Z)\le\frac{1}{2\left(e^{2t}-1\right)}W_2^2(X,Z),
\quad t>0.
\end{equation}

Let $X$ be an $\R^d$-valued random vector.
For $f:\R^d\to\R$ with $\E f(X)^2<\infty$, we write
\[
 \Var(f(X))
= \E\left[\bigl(f(X)-\E f(X)\bigr)^2\right].
\]
We say that $X$ satisfies a \emph{Poincar\'e inequality} if for some $C>0$ 
\begin{equation}\label{eq:poincare}
	\Var(f(X))\le C\E\left[|\nabla f(X)|^2\right]
\end{equation}
holds for all locally Lipschitz functions $f:\R^d\to\R$.
The Poincar\'e constant of $X$ is defined by
\[
C_P(X)
:=\inf\bigl\{C>0:\ \Var(f(X))\le C\E\left[|\nabla f(X)|^2\right]
\ \text{for all locally Lipschitz } f:\R^d\to\R\bigr\}.
\]

Let $\Sigma=\Cov(X)=\E[XX^\top]$.
For any $y\in\R^d$, consider the linear test function $f_y(x)=y\cdot x$.
Then $\nabla f_y(x)=y$ and
\[
\Var(f_y(X))=\Var(y\cdot X)=y^\top \Sigma y,
\qquad \E\big[|\nabla f_y(X)|^2\big]=|y|^2.
\]
Substituting this into \eqref{eq:poincare} gives
\[
y^\top \Sigma y \le C_P(X)|y|^2.
\]
Hence, $C_P(X)$ satisfies the universal lower bound
\begin{equation*}
	C_P(X) \ge \sup_{y\ne0}\frac{y^\top\Sigma y}{|y|^2}
	=\|\Sigma\|_{\mathrm{op}}.
\end{equation*}
In particular, if $\Cov(X)=I_d$, then
\[
C_P(X)\ge 1.
\]

Finally, we state a stability property of the Poincar\'e constant under
OU smoothing, which is a direct consequence of the scaling property $C_P(aX)=a^2C_P(X)$ for $a\in\R$ and the convolution property $C_P(Y_1+Y_2)\leq C_P(Y_1)+C_P(Y_2)$ for independent random variables $Y_1$ and $Y_2$.
\begin{lemma}\label{lem:CP-OU}
	Let $X$ be an $\R^d$-valued random vector with Poincar\'e constant $C_P(X)$ and
	$Z\sim\mathcal N(0,I_d)$ independent of $X$.
	For $t\ge0$, set
	\[
	X(t):=e^{-t}X+
    \sqrt{1-e^{-2t}}Z.
	\]
	Then for every $t\ge0$,
	\[
	C_P(X(t))\le e^{-2t}C_P(X)+(1-e^{-2t}).
	\]
	In particular, if $\Cov(X)=I_d$, then
	\[
	C_P(X(t))\le C_P(X).
	\]
\end{lemma}
\subsection{Ridge function space and additive function space}\label{subsec-space}
    For $n\geq 2$, let $X_1,\dots,X_n$ be i.i.d. $\R^d$-valued random vectors. Set
\begin{equation*} 
  S_m:=\sum_{k=1}^m X_k,\qquad m=1,\dots,n,  
\end{equation*}
then $Z_n=\frac{S_n}{\sqrt n}$.
Consider the Hilbert space
\begin{align*}
	L_0^2(X_1,\dots,X_n;\R^d)
	:=\Bigl\{h(X_1,\dots,X_n)\mid 
	h:(\R^d)^n\to\R^d,\ &\E [h(X_1,\dots,X_n)]=0,\\
	&\E\big[|h(X_1,\dots,X_n)|^2\big]<\infty\Bigr\},
\end{align*}
equipped with the inner product
\[
\langle h_1(X_1,\dots,X_n),h_2(X_1,\dots,X_n)\rangle=\E\bigl[h_1(X_1,\dots,X_n)\cdot h_2(X_1,\dots,X_n)\bigr].
\]
For each $i=1,\dots,n$, define the subspace
\[
L_0^2(X_i;\R^d)
:=\bigl\{h(X_i)\mid h:\R^d\to\R^d,\ \E [h(X_i)]=0,\ \E\big[|h(X_i)|^2\big]<\infty\bigr\}.
\]
The additive subspace and the ridge subspace of $L_0^2(X_1,\dots,X_n;\R^d)$ are defined by
\[
L_{\rm add}^2(X_1,\dots,X_n;\R^d)
:=\left\{\sum_{i=1}^n h_i(X_i)\,\Big|\, h_i\in L_0^2(X_i;\R^d),\ 1\le i\le n\right\},
\]
and
\[
L_{\rm ridge}^2(X_1,\dots,X_n;\R^d)
:=\left\{h(Z_n)\,\big|\,
h:\R^d\to\R^d,\ 
\E [h(Z_n)]=0,\ 
\E\left[|h(Z_n)|^2\right]<\infty\right\}.
\]
Since $X_1,\dots,X_n$ are independent, the subspaces $L_0^2(X_i;\R^d)$,
$1\le i\le n$, are mutually orthogonal in $L_0^2(X_1,\dots,X_n;\R^d)$. Hence,
\[
L_{\rm add}^2(X_1,\dots,X_n;\R^d)=\bigoplus_{i=1}^n L_0^2(X_i;\R^d).
\]
Therefore, for any $H\in L_0^2(X_1,\dots,X_n;\R^d)$, the $L^2$-orthogonal projection of $H$ onto
$L_{\rm add}^2(X_1,\dots,X_n;\R^d)$ is given by
\begin{equation}\label{eq:add-proj}
\Pi_{\rm add}H=\sum_{i=1}^n \E[H\mid X_i].
\end{equation}

\begin{lemma}[{\cite[Lemma  3.1]{JB04}}]\label{lem:score-sum}
	Let $Y_1,Y_2$ be independent $\R^d$-valued random vectors.
 Then 
	\begin{equation}\label{eq-proj-1}
	\rho_{Y_1+Y_2}(Y_1+Y_2)=\E[\rho_{Y_1}(Y_1)\mid Y_1+Y_2]=\E[\rho_{Y_2}(Y_2)\mid Y_1+Y_2].
	\end{equation}
\end{lemma}
In terms of projections, $\rho_{Y_1+Y_2}(Y_1+Y_2)$ is the $L^2$-projection of each individual score $\rho_{Y_i}(Y_i)$ onto the ridge subspace $L^2_{\mathrm{ridge}}(Y_1,Y_2;\R^d)$. Consequently, the projection of the additive score vector $\rho_{Y_1}(Y_1)+\rho_{Y_2}(Y_2)$ onto the same subspace is equal to $2\rho_{Y_1+Y_2}(Y_1+Y_2)$.

 For $n\geq 2$ and $m\in \{2,\dots,n\}$, if we take $Y_1=X_1$ and $Y_2=X_2+\dots+X_m$, then \eqref{eq-proj-1} yields
\begin{equation*}
	\rho_{S_m}(S_m)=\E[\rho_{X_1}(X_1)\mid S_m].
	\end{equation*}
For any measurable function $\phi$ on $\R^d$, it follows that \begin{equation}\label{eq-proj-m}
	\E[\phi(S_m)\rho_{X_1}(X_1)]=\E[\phi(S_m)\rho_{S_m}(S_m)].
	\end{equation}

\section{Proof of the main theorem}\label{sec-main}
\subsection{Johnson--Barron projection method and related lower bounds}\label{subsec-lowerbound}
For $n\geq 2$, let $X_1,\dots,X_n$ be i.i.d. $\R^d$-valued random vectors. Recall $S_m=X_1+\dots+X_m$ and $Z_n=\frac{S_n}{\sqrt{n}}$ for $m, n \in \mathbb{N}$. For $f:\R^d\to\R^d$ such that $f(Z_n)\in L_{\rm ridge}^2(X_1,\dots,X_n;\R^d)$, we define recursively $(f_m)_{m=1}^n$ by 
\begin{equation} \label{e:fnSeq}
\begin{split}
&f_n(\cdot)=f(\cdot), \\
&f_m(\cdot)=\E_{X_{m+1}}\Bigl[f_{m+1}\Bigl(\cdot+\frac{X_{m+1}}{\sqrt n}\Bigr)\Bigr], \ \ \ m=1,2,\dots,n-1.    
\end{split}
\end{equation}
For each $m=1,\dots,n$,
\[
f_m\Bigl(\frac{S_m}{\sqrt n}\Bigr)
=\E\bigl[f(Z_n)\mid S_m\bigr].
\]
We also define, for $u\in\R^d$,
\[
g(u):=\sqrt n\E\Bigl[f\Bigl(\frac{S_{n-1}+u}{\sqrt n}\Bigr)\Bigr],
\]
so that
\[
g(X_i)=\sqrt n\E\bigl[f(Z_n)\mid X_i\bigr],\qquad i=1,\dots,n.
\]
It is clear to see from \eqref{eq:add-proj} that
$$\Pi_{\rm add} f(Z_n)=\frac{1}{\sqrt n}\sum_{i=1}^n g(X_i).$$
\begin{lemma} \label{lem:successive-tele}
Let $X_1,\dots,X_n$ be i.i.d.\ $\R^d$-valued random vectors with absolutely continuous densities.
	Assume $\E X_1=0$, $\Cov(X_1)=I_d$, and that $X_1$ has finite Poincar\'e constant $R$ and finite Fisher information $I(X_1)$. 
	Under the assumptions and notation above, 
	\begin{equation} \label{e:ProjError1}
\E\left[\Bigl|f\Bigl(Z_n\Bigr)-\frac{1}{\sqrt n}\sum_{i=1}^n g(X_i)\Bigr|^2\right]=\sum_{i=2}^n \Delta_i,
	\end{equation}
    where $\Delta_i:=\E\left[\Bigl|
	f_i\Bigl(\frac{S_i}{\sqrt n}\Bigr)-f_{i-1}\Bigl(\frac{S_{i-1}}{\sqrt n}\Bigr)-\frac1{\sqrt n}g(X_i)\Bigr|^2\right]$ with the following lower bound:
	\begin{equation} \label{e:ProjError2}
	\Delta_i\ge\frac{i-1}{nI(X_1)R}\E\left[|g(X_1)-MX_1|^2\right], \quad \quad  M=\mathbb{E}[\nabla g(X_1)].
	\end{equation}
    In particular, as $f(\cdot)=\rho_{Z_n}(\cdot)$, we have 
    \begin{equation} \label{e:M=FMatrix}
    M=-I_M(Z_n),
    \end{equation}
    where $I_M(Z_n)$ is the Fisher information matrix of $Z_n$ defined by \eqref{e:FIMatrix}.  
\end{lemma}
\begin{proof}
Let us first prove \eqref{e:ProjError1}.
For $m=1,\dots,n$, denote
$$I_m=f_m\Bigl(\frac{S_m}{\sqrt n}\Bigr)-\frac1{\sqrt n}\sum_{i=1}^m g(X_i), \quad \delta_m=f_m\Bigl(\frac{S_m}{\sqrt n}\Bigr)-f_{m-1}\Bigl(\frac{S_{m-1}}{\sqrt n}\Bigr)-\frac1{\sqrt n}g(X_m).$$
It is clear to see that for $m=2,3,\dots,n$,
\begin{align*}
I_m=I_{m-1}+\delta_m, \quad \ \Delta_m=\E\big[|\delta_m|^2\big]. 
\end{align*}
We have 
\begin{equation} \label{e:DecStep1}
	\E \big[|I_n|^2\big]=\E \big[|I_{n-1}|^2\big]+\Delta_n+\E[I_{n-1} \cdot \delta_n]
    =\E \big[|I_{n-1}|^2\big]+\Delta_n,
\end{equation}
since $\E[I_{n-1} \cdot \delta_n]=0$ by the following argument: we have $\E[\delta_n|X_1,\dots,X_{n-1}]=0$ by a straightforward computation and thus 
\begin{equation*}
\E[I_{n-1} \cdot \delta_n]=\E\big[\E[\delta_n|X_1,\dots,X_{n-1}]\cdot I_{n-1}\big]=0.
\end{equation*}
Repeating the argument \eqref{e:DecStep1}, we obtain 
\begin{align*}
	\E \big[|I_n|^2\big]=\E \big[|I_{n-1}|^2\big]+\Delta_n=\dots=\E\big[|I_1|^2\big]+\sum_{i=2}^n \Delta_i.
\end{align*}
Combining this with the fact that
$$
\E\big[|I_1|^2\big]=\E\left[\Bigl|f_1\Bigl(\frac{S_1}{\sqrt n}\Bigr)-\frac1{\sqrt n}g(X_1)\Bigr|^2\right]=0,
$$ 
we immediately obtain \eqref{e:ProjError1}. 

Let us now prove \eqref{e:ProjError2}. 
	Write $f=(f^{(1)},\dots,f^{(d)})$ and $g=(g^{(1)},\dots,g^{(d)})$. For $j=1,\dots,d$, set
	\[	\Delta_i^{(j)}:=\E\left[\Bigl(f_i^{(j)}\Bigl(\frac{S_i}{\sqrt n}\Bigr)-f_{i-1}^{(j)}\Bigl(\frac{S_{i-1}}{\sqrt n}\Bigr)-\frac1{\sqrt n}g^{(j)}(X_i)\Bigr)^2\right],
	\]
	so that $\Delta_i=\sum_{j=1}^d \Delta_i^{(j)}$. Similarly, we have $$\E\left[|g(X_1)-MX_1|^2\right]=\sum_{j=1}^d\E\left[\big(g^{(j)}(X_1)-\E[\nabla g^{(j)}(X_1)]\cdot X_1\big)^2\right].$$
Therefore, it suffices to prove the lower bound for each coordinate
$j\in\{1,\dots,d\}$ and then sum over $j$.

\emph{Step 1. Auxiliary function and Cauchy--Schwarz bound}:
    For a fixed $1 \le j \le d$ and all $2 \le i \le n$, define the $\mathbb{R}^d$-valued function, for $u \in \mathbb{R}^d$, 
	\begin{equation*}
	r_j(u):=\E\Bigl[\Bigl(f_i^{(j)}\Bigl(\frac{S_{i-1}+u}{\sqrt n}\Bigr)-f_{i-1}^{(j)}\Bigl(\frac{S_{i-1}}{\sqrt n}\Bigr)-\frac1{\sqrt n}g^{(j)}(u)\Bigr)\sum_{k=1}^{i-1}\rho_{X_k}(X_k)\Bigr].
	\end{equation*}
	By the Cauchy--Schwarz inequality, the i.i.d. property, and $\E[\rho_{X_k}(X_k)]=0$, we get
	\begin{equation*}
		\E\left[|r_j(X_i)|^2\right]
		\le \Delta_i^{(j)}\E\left[\Bigl|\sum_{k=1}^{i-1}\rho_{X_k}(X_k)\Bigr|^2\right]
		= (i-1) \Delta_i^{(j)} I(X_1),
	\end{equation*}
and thus 
\begin{equation}\label{eq:succproj-Cauchy}
		\Delta_i^{(j)} \ge \frac{\E\left[|r_j(X_i)|^2\right]}{(i-1)I(X_1)}.
	\end{equation}

   \emph{Step 2. A simple form of $r_j(u)$ by exchangeability}:
		By exchangeability of $X_1,\dots,X_{i-1}$ and $\E[\rho_{X_k}(X_k)]=0$, we have
	\begin{align}\label{eq:rj-sym}
		r_j(u)
		= (i-1)\E\Bigl[\Bigl(f_i^{(j)}\Bigl(\frac{S_{i-1}+u}{\sqrt n}\Bigr)-f_{i-1}^{(j)}\Bigl(\frac{S_{i-1}}{\sqrt n}\Bigr)\Bigr)\rho_{X_1}(X_1)\Bigr].
	\end{align}
	From the definition of $g$ and the exchangeability,
	\[
	\frac1{\sqrt n}g^{(j)}(u)=\E \left[f_i^{(j)}\Bigl(\frac{S_{i-1}+u}{\sqrt n}\Bigr)\right].
	\]
	Applying Lemma \ref{lem:score-shift} to $S_{i-1}$ and $f_i^{(j)}(\cdot/\sqrt n)$ yields
	\begin{equation*}
		\nabla\Bigl(\frac1{\sqrt n}g^{(j)}(u)\Bigr)
		=-\E\Bigl[f_i^{(j)}\Bigl(\frac{S_{i-1}+u}{\sqrt n}\Bigr)\rho_{S_{i-1}}(S_{i-1})\Bigr].
	\end{equation*}
	Thus, by \eqref{eq-proj-m}, the first term in \eqref{eq:rj-sym} is
	\begin{equation}\label{eq:first-term-result}
    \begin{aligned}
		(i-1)\E\Bigl[f_i^{(j)}\Bigl(\frac{S_{i-1}+u}{\sqrt n}\Bigr)\rho_{X_1}(X_1)\Bigr]&=(i-1)\E\Bigl[f_i^{(j)}\Bigl(\frac{S_{i-1}+u}{\sqrt n}\Bigr)\rho_{S_{i-1}}(S_{i-1})\Bigr]\\&=-(i-1)\nabla\Bigl(\frac1{\sqrt n}g^{(j)}(u)\Bigr).
         \end{aligned}
	\end{equation}
For the second term in \eqref{eq:rj-sym}, we have
	\begin{equation}
\label{eq:second-term-result}
\begin{split}
(i-1)\E\Bigl[f_{i-1}^{(j)}\Bigl(\frac{S_{i-1}}{\sqrt n}\Bigr)\rho_{X_1}(X_1)\Bigr] 
&
=(i-1)\E\Bigl[f_i^{(j)}\Bigl(\frac{S_{i-1}+X_i}{\sqrt n}\Bigr)\rho_{X_1}(X_1)\Bigr]\\
&=-\frac{i-1}{\sqrt n}\E[\nabla g^{(j)}(X_1)].
\end{split}
	\end{equation}
 Inserting \eqref{eq:first-term-result} and \eqref{eq:second-term-result} into \eqref{eq:rj-sym}, we obtain
	\begin{equation*}
		r_j(u)=-\frac{i-1}{\sqrt n}\bigl(\nabla g^{(j)}(u)-\E[\nabla g^{(j)}(X_1)]\bigr).
	\end{equation*}
    
\emph{Step 3. A lower bound via the Poincar\'e inequality and the final sum}:
    By the simple form of $r_j(u)$, we have
	\begin{equation*}
		\E\left[|r_j(X_i)|^2\right]=\frac{(i-1)^2}{n}\E\left[|\nabla g^{(j)}(X_1)-\E[\nabla g^{(j)}(X_1)]|^2\right].
	\end{equation*}
	Together with \eqref{eq:succproj-Cauchy}, this implies
	\[
	\Delta_i^{(j)}\ge\frac{i-1}{nI(X_1)}\E\left[|\nabla g^{(j)}(X_1)-\E[\nabla g^{(j)}(X_1)]|^2\right].
	\]
	Since $\E [g(X_1)]=\sqrt n\E [f(Z_n)]=0$ and $\E X_1=0$, we have $$\E[g^{(j)}(X_1)-\E[\nabla g^{(j)}(X_1)]\cdot X_1]=0.$$
	Applying the Poincar\'e inequality to $\varphi_j(u):=g^{(j)}(u)-\E[\nabla g^{(j)}(X_1)]\cdot u$ yields
	\[
	\E\left[\bigl(g^{(j)}(X_1)-\E[\nabla g^{(j)}(X_1)]\cdot X_1\bigr)^2\right]
	\le R\E\left[|\nabla g^{(j)}(X_1)-\E[\nabla g^{(j)}(X_1)]|^2\right].
	\]
	Therefore,
	\[
	\Delta_i^{(j)}\ge\frac{i-1}{nI(X_1)R}\E\left[\big(g^{(j)}(X_1)-\E[\nabla g^{(j)}(X_1)]\cdot X_1\big)^2\right].
	\]
	Summing over $j=1,\dots,d$ gives
	\[
	\Delta_i\ge\frac{i-1}{nI(X_1)R}\E\left[|g(X_1)-M X_1|^2\right],
	\]
where $M=\mathbb{E}[\nabla g(X_1)]$.

It remains to show as $f(\cdot)=\rho_{Z_n}(\cdot)$, we have $M=\mathbb{E}[\nabla g(X_1)]=-I_M(Z_n)$. Indeed, by \eqref{eq:second-term-result} with $i=2$ and noticing $g(X_1)=\sqrt{n} f_1(\frac{X_1}{\sqrt n})$, we have
\begin{equation*}
M=\E [\nabla g(X_1)]=-\E [g(X_1)\rho_{X_1}(X_1)^\top] =-\sqrt n\E[\rho_{Z_n}(Z_n)\rho_{X_1}(X_1)^\top].
\end{equation*}
This, together with \eqref{eq-proj-m} and the relation $\rho_{Z_n}(Z_n)=\sqrt{n} \rho_{S_n}(S_n)$, gives 
\begin{equation*}
\begin{split}
M=-\sqrt n\E[\rho_{Z_n}(Z_n)\rho_{S_n}(S_n)^\top]=-\E[\rho_{Z_n}(Z_n)\rho_{Z_n}(Z_n)^\top]=-I_M(Z_n), 
\end{split}
\end{equation*}
which completes the proof.
\end{proof}

	\subsection{Decay of the relative Fisher information}\label{subsec-FI}

    	\begin{proposition}\label{Prop-FI}
		Let $X_1,X_2,\dots,X_n$ be i.i.d.\ $\R^d$-valued random vectors with absolutely continuous density,
		$\E X_1=0$, $\Cov(X_1)=I_d$, and Poincar\'e constant $C_P(X_1)=R$. 
		Let $Z_n=\frac1{\sqrt n}\sum_{i=1}^n X_i$. Then for all $n\geq1$,
		\[
		J(Z_n)\leq \frac{2dR}{2dR+(n-1)}J(X_1).
		\]
	\end{proposition}

\begin{remark}
As noted in Remark \ref{comparem}\,\ref{compadfh3}, recent progress on the slicing problem yields dimension-dependent bounds for the relative entropy. More precisely, if  $X_1$ is  log-concave, it holds from \cite[Equation (3) and Theorem 1.2]{KL25} that
   \[
	\Ent(X_1\mid Z)\le Cd.
	\]  
Unlike the relative entropy $\Ent(X_1\mid Z)$, however, we can not expect a general and explicit upper bound for $J(X_1)$ in $d$, see Examples \ref{exam1} and \ref{exam2} below.
\end{remark}

\begin{example}\label{exam1}
For $\beta>1$, let $X_\beta$ be the one-dimensional random variable with density
\[
q_\beta(x)=c_\beta\exp\left(-b_\beta|x|^\beta\right),
\]
where
\[c_\beta:= \frac{\beta}{2\Gamma(1/\beta)} \sqrt{\frac{\Gamma(3/\beta)}{\Gamma(1/\beta)}},\quad b_\beta:=\left(\frac{\Gamma(3/\beta)}{\Gamma(1/\beta)}\right)^{\beta/2}.
\]
Then $\E X_\beta=0$, $\Var(X_\beta)=1$, and \[
J(X_\beta)
=\beta^2\frac{\Gamma(3/\beta)\Gamma(2-1/\beta)}{\Gamma(1/\beta)^2}-1.
\]
As $\beta \to \infty$, it follows that $J(X_\beta) \sim \beta/3$.  
For each $d\ge1$, let $X=(X_{\beta}^{(1)},\dots,X_{\beta}^{(d)})$ have i.i.d. coordinates $X_{\beta}^{(i)}\overset{Law}{=} X_{\beta}$, $i=1,\dots,d$. Then
$\E X=0$, $\Cov(X)=I_d$ and $X$ is log-concave.
By the additivity of the relative Fisher information for product measures,
\[
J(X)
=dJ(X_{\beta})
\sim \frac{\beta}{3}d,\quad \beta \to\infty.
\]
By choosing  $\beta=\beta(d)\sim  3d$ as $d\to\infty$, we obtain $J(X)\sim d^2$. More generally, by allowing $\beta(d)$ to grow faster than linearly in $d$, one can make
$J(X)$ grow faster than quadratically.
\end{example}
It is seen from Example \ref{exam1} that, even within the class of centered, isotropic, log-concave random variables, the relative Fisher information can grow linearly or superlinearly with respect to the dimension. Since $\Gamma(t)=1/t-\gamma+o(1)$ as $t\downarrow0$, where $\gamma=0.5772\dots$ is the Euler--Mascheroni constant, it is not hard to verify that
\[
        \lim_{\beta\to\infty}q_\beta(x)=
        \begin{cases}
        (2\sqrt{3})^{-1},&|x|<\sqrt{3},\\
        (2\sqrt{3})^{-1}\exp\left(-e^{-\gamma}\right),&|x|=\sqrt{3},\\
        0,& |x|>\sqrt{3}.
        \end{cases}
\]
In other words, as $\beta\to\infty$, the pointwise limit of $q_\beta$ can be regarded as the uniform density on $[-\sqrt{3},\sqrt{3}]$.
Thus, Example \ref{exam1} presents a sequence of smooth log-concave laws approaching the uniform distribution, whose relative Fisher information is
infinite.

  Next, we give another example in which the relative Fisher information grows arbitrarily fast in $d$. The construction is based on Student's $t$-distribution, which is not log-concave.

\begin{example}\label{exam2}
For $\theta>2$, let $X_\theta$ be the one-dimensional random variable with density
\[
f_\theta(x)
=c_\theta\Bigl(1+\frac{x^2}{\theta-2}\Bigr)^{-(\theta+1)/2},
\quad
c_\theta:=\frac{\Gamma((\theta+1)/2)}{\sqrt{\pi(\theta-2)}\Gamma(\theta/2)}.
\]
Then $\E X_\theta=0$, $\Var(X_\theta)=1$,
and 
\[
J(X_\theta)=\frac{6}{(\theta-2)(\theta+3)}.
\]
Note that $J(X_\theta)$ is strictly decreasing in $\theta\in(2,\infty)$ with $J(X_\theta)\to\infty$ as $\theta\downarrow2$ and
$J(X_\theta)\to0$ as $\theta\to\infty$.
Therefore, for any $k>0$, we can choose $\theta(k)$ such that $J(X_{\theta(k)})=k$. 
For each $d\ge1$, let $X=(X_{\theta(k)}^{(1)},\dots,X_{\theta(k)}^{(d)})$ have i.i.d. coordinates $X_{\theta(k)}^{(i)}\overset{Law}{=} X_{\theta(k)}$, $i=1,\dots,d$. Then
$\E X=0$, $\Cov(X)=I_d$, and
\[
J(X)=kd.
\]
Since $k>0$ is arbitrary, by allowing it to depend on $d$, the relative Fisher information $J(X)$ can be made to exhibit any prescribed dependence on the dimension $d$.
\end{example}

\begin{proof}[Proof of Proposition \ref{Prop-FI}]
We shall follow the elegant Johnson--Barron projection method and split the proofs into the following steps. 

	\emph{Step 1. Projection to the ridge field:}
Define the additive field $U$ and the ridge field $V$ as follows:
\[ U = \frac{1}{\sqrt{n}} \sum_{i=1}^{n} \big(\rho_{X_{i}}(X_{i}) + X_{i}\big), \quad \quad V = \rho_{Z_{n}}(Z_{n}) + Z_{n}, \]
where $Z_n=\frac{X_1+X_2+\dots+X_n}{\sqrt{n}}$. 
By the i.i.d. and mean zero property, we have
$$\mathbb{E}[|U|^{2}] = J(X_{1}), \quad \quad
    \mathbb{E}[|V|^{2}] = J(Z_{n}).$$
By Lemma \ref{lem:score-sum} with straightforward computation, we get $$V = \mathbb{E}[U \mid Z_{n}],$$ identifying $V$ as the orthogonal projection of $U$ onto the ridge subspace $L_{ridge}^{2}$. From the properties of projections, it is immediate that $\mathbb{E}[|V|^{2}] \le \mathbb{E}[|U|^{2}]$, yielding the monotonic decay $$J(Z_{n}) \le J(X_{1}).$$ 

\emph{Step 2. Projection of ridge field:}
We now project $V$ back onto the additive subspace $L_{add}^{2}$. Let $\widehat{V} = \Pi_{add} V$, which reads as
\[ \widehat{V} = \sum_{i=1}^{n} \mathbb{E}[V \mid X_{i}].\]
By the symmetry of the i.i.d. summands, each term $\mathbb{E}[V \mid X_{i}]$ takes the same form. Recalling that $g(X_{i}) = \sqrt{n} \mathbb{E}[\rho_{Z_{n}}(Z_{n}) \mid X_{i}]$, then
\[ \widehat{V} = \frac{1}{\sqrt{n}} \sum_{i=1}^{n} (g(X_{i}) + X_{i}). \]
The difference $V - \widehat{V}$ is orthogonal to the additive subspace by construction. Furthermore, because $V$ is a projection of $U$, the geometry dictates that 
\[
	\langle U,\widehat V\rangle=\langle U,V\rangle=\langle V,V\rangle=\E[|V|^2]=J(Z_n).
	\]
Applying the Cauchy--Schwarz inequality, we obtain a crucial lower bound on the norm of the additive projection:
\begin{equation} \label{e:HatVEst}
 \mathbb{E}[|\widehat{V}|^{2}] \ge \frac{\langle U, \widehat{V} \rangle^{2}}{\mathbb{E}[|U|^{2}]} = \frac{J(Z_{n})^{2}}{J(X_{1})}.
\end{equation}

\emph{Step 3. Decomposition of the error $V - \widehat{V}$ and an inequality of $J(Z_n)$:} It is obvious that 
\[
V-\widehat V
=\rho_{Z_n}(Z_n)-\frac1{\sqrt n}\sum_{i=1}^n g(X_i).
\]
Recall \eqref{e:fnSeq} and define the functions recursively by 
$$f_{n}(\cdot)= \rho_{Z_{n}}(\cdot),$$ and 
$$f_{m}(\cdot) = \mathbb{E}\left[f_{m+1}\left(\cdot + \frac{X_{m+1}}{\sqrt{n}}\right)\right], \quad \quad m=n-1,\dots,1.$$ 
On the one hand, by Lemma \ref{lem:successive-tele}, we have
\begin{equation*}
\begin{split}
	\E\big[|V-\widehat V|^2\big]
	& =\E\left[\Big|\rho_{Z_n}(Z_n)-\frac1{\sqrt n}\sum_{i=1}^n g(X_i)\Big|^2\right]  \\
 &\ge \sum_{i=2}^n  \frac{i-1}{n I(X_1)R}\,
\E\left[|g(X_1)-M X_1|^2\right] \\
& \ge \frac{n-1}{2R I(X_1)}\E\left[|g(X_1)-M X_1|^2\right],
\end{split}
\end{equation*}
where we recall $M= -I_M(Z_n)\in\R^{d\times d}$. 
On the other hand, 
Pythagoras' Theorem for the orthogonal projection of $V$ onto
	$L^2_{\rm add}$ then yields the upper bound
	\begin{equation*}
		\E\big[|V-\widehat V|\big]^2
		=\E[|V|^2]-\E[|\widehat V|^2]
		\le J(Z_n)-\frac{J(Z_n)^2}{J(X_1)}, 
 \end{equation*}
where the last inequality is by \eqref{e:HatVEst}. 
Hence,
\begin{equation} \label{e:JnTwoSides}
J(Z_n)-\frac{J(Z_n)^2}{J(X_1)} \ge \frac{n-1}{2R I(X_1)}\E\left[|g(X_1)-M X_1|^2\right].
\end{equation}

\emph{Step 4. Estimate of $\E|g(X_1)-M X_1|^2$:}
By the definition of $g$, the symmetry of the summand and Lemma \ref{lem:score-moments}, we get
\begin{equation*}
\begin{split}
\E\bigl[X_1 g(X_1)^\top\bigr]
&=\sqrt n\E\bigl[X_1\rho_{Z_n}(Z_n)^\top\bigr]=\sqrt n\E\left[\frac{X_1+\dots+X_n}n\rho_{Z_n}(Z_n)^\top\right]\\
&=\E\left[Z_n\rho_{Z_n}(Z_n)^\top\right]=-I_d.
\end{split}
\end{equation*}
Since $\operatorname{Cov}(X_1)=I_d$, we have
\begin{equation*}
\E\bigl[(g(X_1)+X_1)X_1^\top\bigr]=0,
\end{equation*}
which immediately implies the following $L^2$-decomposition:
\begin{equation}\label{eq-decomposition}
    \E\left[|g(X_1)-M X_1|^2\right]
=\E\left[|g(X_1)+X_1|^2\right]
+\E\left[|(M+I_d)X_1|^2\right].
\end{equation}
By the definition of $\hat V$ and \eqref{e:HatVEst}, we obtain
\begin{equation}\label{eq:lower-part1}
	\E\left[|g(X_1)+X_1|^2\right]
	= \E[|\widehat V|^2]
	\ge \frac{J(Z_n)^2}{J(X_1)}.
\end{equation}
Since $\E[X_1]=0$ and $\Cov(X_1)=I_d$, we get
\[
\E\left[|(M+I_d)X_1|^2\right]
=\tr\bigl((M+I_d)^\top(M+I_d)\bigr)
=\|M+I_d\|_F^2,
\]
where $\|A\|_F:=(\tr(A^\top A))^{1/2}$ is the Frobenius norm.  Applying the Cauchy--Schwarz inequality in the form $\|A\|_F^2 \ge (\tr A)^2/d$ and recalling that $M = -I_M(Z_n)$, we find
\begin{equation}\label{eq:lower-part2}
	\|M+I_d\|_F^2
	\ge \frac{\left(\tr(M+I_d)\right)^2}{d}= \frac{J(Z_n)^2}{d}, 
\end{equation} 
where the last equality follows from the fact that 
$$J(Z_n)=I(Z_n)-d=-\tr(M)-d.$$ 

Putting \eqref{eq-decomposition}, \eqref{eq:lower-part1} and \eqref{eq:lower-part2} together, we obtain
\begin{equation} \label{e:g-MBound}
\E\left[|g(X_1)-M X_1|^2\right]
\ge \frac{J(Z_n)^2}{J(X_1)}+\frac{J(Z_n)^2}{d}.
\end{equation}

\emph{Step 5. Conclusion of the proof:}
Combining \eqref{e:g-MBound} and \eqref{e:JnTwoSides}, we get
\[
J(Z_n)-\frac{J(Z_n)^2}{J(X_1)}
 \ge \frac{n-1}{2R I(X_1)}
\left(\frac{J(Z_n)^2}{J(X_1)}+\frac{J(Z_n)^2}{d}\right)=\frac{n-1}{2RdJ(X_1)}J(Z_n)^2,
\]
and thus Proposition \ref{Prop-FI} follows immediately.
\end{proof}

	\subsection{Proof of Theorem \ref{thm-1}  }\label{subsec-ent}	
	\ To prove Theorem \ref{thm-1} based on Proposition \ref{Prop-FI}, we need the following lemma on the CLT in Wasserstein-2 distance.
\begin{lemma}[{\cite[Theorem 6]{CFP19}}]\label{lem:W2-Poincare}
	Let $X_1,\dots,X_n$ be i.i.d. $\R^d$-valued random vectors with $\E X_1=0$ and $\Cov (X_1)= I_d$.
	Assume that the law of $X_1$ satisfies a Poincar\'e inequality with constant $C_P(X_1)=R$, then
	\[
	W_2^2(Z_n,Z)\le \frac{d(R-1)}{n}.
	\]
\end{lemma}
	\begin{proof}[Proof of  Theorem \ref{thm-1}]
	For $t\ge0$, let
	\[
	X_i(t):=e^{-t}X_i+\sqrt{1-e^{-2t}}Z_i,\qquad i=1,\dots,n,
	\]
	where $Z_i\sim\mathcal N(0,I_d)$ are i.i.d. and independent of
	$\{X_i\}$.  Set
	\[
	Z_n(t):=\frac1{\sqrt n}\sum_{i=1}^n X_i(t).
	\]
	Since $\Cov(X_1)=I_d$, we have $\Cov(X_1(t))=\Cov(Z_n(t))=I_d$ for all $t\ge0$.  By
	Lemma \ref{lem:CP-OU}, $X_1(t)$ satisfies the Poincar\'e inequality
with  Poincar\'e constant
	\[
	C_P(X_1(t))\le C_P(X_1)=R,\qquad t\ge0.
	\]
	Applying Proposition \ref{Prop-FI} to 
	$X_1(t),\dots,X_n(t)$, we obtain that
\begin{equation}\label{eq:J-pointwise-t}
		J(Z_n(t))\le \frac{2dR}{2dR+(n-1)}J(X_1(t)).
	\end{equation}
	Letting $t\to\infty$ in Lemma \ref{lem:debruijn} yields
	\[
	\Ent(X_1\mid Z)=\int_0^\infty J(X_1(s))\dd s,
	\qquad
	\Ent(Z_n\mid Z)=\int_0^\infty J(Z_n(s))\dd s.
	\]
	Combining with
	\eqref{eq:J-pointwise-t}, we get
	\begin{equation}\label{ent-1}
    \begin{aligned}
        \Ent(Z_n\mid Z)
	=\int_0^\infty J(Z_n(t))\dd t
	&\le \frac{2dR}{2dR+(n-1)}\int_0^\infty J(X_1(t))\dd t
	\\&= \frac{2dR}{2dR+(n-1)}\Ent(X_1\mid Z).
    \end{aligned}
	\end{equation}
 
 For $t>0$, applying \eqref{eq:Ent-W2-FI-intro} with $X=Z_n$ gives
			\[
			\Ent(Z_n\mid Z)
	\le \frac{1}{2(e^{2t}-1)}W_2^2(Z_n,Z)
			+ \frac{1-e^{-2t}}{2}J(Z_n).
			\]
	We consider	\[
		\Phi(t)
		:=\frac{1}{2(e^{2t}-1)}W_2^2(Z_n,Z)
		+ \frac{1-e^{-2t}}{2}J(Z_n),\quad t\in(0,\infty).
		\]
	Since $W_2^2(Z_n,Z)\le J(Z_n)$ (\cite{Tal96, Gro75}), optimizing $	\Phi(t)$ on $t\in (0,\infty)$ yields
	\[
\inf _{ t\in(0,\infty)}	\Phi(t)=W_2(Z_n,Z)\sqrt{J(Z_n)}-\frac12 W_2^2(Z_n,Z).	\] 
		This implies the HWI inequality (cf.\ \cite[Theorem 3]{OV00})
		\[
		\Ent(Z_n\mid Z)
		\le W_2(Z_n,Z)\sqrt{J(Z_n)}-\frac12 W_2^2(Z_n,Z)\leq W_2(Z_n,Z)\sqrt{J(Z_n)}.
			\]
			Combining with Lemma \ref{lem:W2-Poincare}
		and Proposition \ref{Prop-FI}, we obtain
			\begin{equation}\label{ent-2}
			\Ent(Z_n\mid Z)
		\le \sqrt{\frac{d(R-1)}{n}\cdot
			\frac{2dR}{2dR+(n-1)}J(X_1)}.    \end{equation}
		Theorem \ref{thm-1} follows from \eqref{ent-1} and \eqref{ent-2}.
		\end{proof}

\noindent
{\bf Acknowledgement}:
C.-S.\ Deng is supported by  the National Natural Science Foundation of China (12371149). L.\ Xu is supported by the Science and Technology Development Fund (FDCT) of Macau S.A.R. FDCT 0074/2023/RIA2.


\end{document}